\documentclass[12pt,a4paper]{amsart}
\usepackage{geometry}
\usepackage{amsmath,amssymb}
\usepackage{amsfonts}
\usepackage{eucal}
\usepackage{amsthm}
\usepackage{mathrsfs}
\numberwithin{equation}{section}
\usepackage{comment}
\usepackage{xcolor}

\def\a{\alpha}

\def\c{\gamma}
\def\d{\delta}
\def\e{\varepsilon}
\def\f{\varphi}
\def\g{\psi}

\def\l{\lambda}
\def\m{\mu}
\def\n{\nu}
\def\o{\omega}

\def\s{\sigma}

\def\x{\xi}

\def\z{\zeta}

\newcommand{\sI}{\mathscr{I}}

\def\re{\mathbb{R}}

\def\pa{\partial}

\newcommand{\supp}{\text{{\rm supp}\;}}

\newtheorem{thm}{Theorem}[section]
\newtheorem{lem}[thm]{Lemma}
\newtheorem{prop}[thm]{Proposition}
\newtheorem{cor}[thm]{Corollary}

\theoremstyle{definition}

\theoremstyle{remark}
\newtheorem{rem}[thm]{Remark}

\title[]{Dispersive estimates and optimality for Schr\"odinger equations on product cones}
\author{Kouichi Taira}
\address{Faculty of Mathematics, Kyushu University, 744, Motooka, Nishi-ku, Fukuoka, Japan}
\email{taira.kouichi.800@m.kyushu-u.ac.jp}
\date{}
\begin{document}
\maketitle

\begin{abstract}
In this paper, we study time decay estimates for the Schr\"odinger propagator on the product cone $(X,g)$, where $X=C(\rho\mathbb{S}^{n-1})=(0,\infty)\times \rho\mathbb{S}^{n-1}$. We prove that the usual dispersive estimate holds when the radius $\rho$ is greater than or equal to $1$ and fails otherwise. A part of the former result was already established in a recent paper by Jia-Zhang. The method used here relies purely on harmonic analysis, whereas Jia-Zhang employed microlocal analysis to capture the precise asymptotic behavior of the propagator.

\end{abstract}


\section{Introduction}

\subsection{Main results}

We consider the product cone $(X,g)$, where $X=C(\rho\mathbb{S}^{n-1})=(0,\infty)\times \rho\mathbb{S}^{n-1}$ ($\rho>0$) and $g=dr^2+r^2g_{\rho\mathbb{S}^{n-1}}$ and the Friedrichs self-adjoint extension of the operator
\begin{align*}
H=-\Delta_g+\frac{c}{r^2},
\end{align*}
defined on $C_c^{\infty}(X)$, where $c$ is a constant satisfying the subcritical condition
\begin{align}\label{eq:csubcri}
c>-\frac{(n-2)^2}{4}.
\end{align}
In this paper, we study the Schr\"odinger equation
\begin{align*}
\left\{\begin{aligned}
&i\pa_tu(t)+Hu(t)=0\\
&u(0)=u_0\in L^2(X)
\end{aligned}\right.
\end{align*}
and a certain time decay property of the propagator $e^{itH}$.

The purpose of this note is the following:

\begin{itemize}
\item To recover and extend the results of the recent paper \cite{JZ} for simple cases $X=C(\rho\mathbb{S}^{n-1})$ in a more elementary way, as done in \cite{TT} for $(k,a)$-generalized Laguerre operators.

\item To deduce the dispersive estimates on higher-dimensional Euclidean spaces (which correspond to the case $\rho=1$), extending the three-dimensional result by Fanelli-Felli-Fontelos-Primo \cite{FFFP1,FFFP2}.

\item To show that the dispersive estimate fails when the conjugate radius $\rho \pi$ of the sphere $\rho\mathbb{S}^{n-1}$ is less than $\pi$, a case not considered in \cite{JZ}.

\end{itemize}
In \cite{JZ}, the authors consider the general product cone $X=C(Y)$ with a closed Riemannian manifold $(Y,h)$ and a more general potential $V(r,y)=r^{-2}V(y)$. Their results include the dispersive estimate 
\begin{align*}
\|e^{itH}\|_{L^1\to L^{\infty}}\lesssim |t|^{-\frac{n}{2}}
\end{align*}
under the assumption that the conjugate radius of $(Y,h)$ is strictly greater than $\pi$ and $V=0$ (or under a more general assumption on $V$ corresponding to the condition $c\geq 0$).
We note that in their future work (see \cite[page. 3]{JZ}), the assumption on the conjugate radius (somehow) may be relaxed. On one hand, we only use elementary tools such as the stationary phase theorem, following the strategy developed in \cite{TT} here. On the other hand, the method used in \cite{JZ} depends heavily on the technology of microlocal analysis to describe the integral kernel of the wave propagator $\cos t\sqrt{-\Delta_Y}$ (or the spectral projection), which in turn, allows them to handle more general classes of closed manifolds $(Y,h)$. The previous paper \cite{Z} by Zhang for the two-dimensional case is written in a similar spirit to that of this paper.

Let us briefly review some related results. The dispersive estimate on two-dimensional metric cones was proved in \cite{F} for all radii $\rho>0$, and an asymptotic expansion of the propagator with respect to the parameter $r_1r_2/2t$ was given there. This result was extended to higher dimensional cases in \cite{JZ}, where the sphere $\rho \mathbb{S}^{n-1}$ was replaced by a more general closed manifold $(Y,h)$ whose conjugate radius is strictly greater than $\pi$.
For the three-dimensinoal Euclidean case $\rho=1$, Fanelli-Felli-Fontelos-Primo \cite{FFFP1,FFFP2} established the dispersive estimate for $c\geq 0$ and some weighted $L^p-L^{p^*}$ estimates even for $c<0$. The dispersive estimate with a fixed angular momentum is proved in \cite{KM}.
In \cite{W}, Wang studied certain decay estimates and asymptotic expansions of the propagator between weighted $L^2$ spaces, which are an extension of the results by Jensen-Kato \cite{JK}.

We write
\begin{align*}
\n_0:=&\sqrt{\left(\frac{n-2}{2}\right)^2+c}.
\end{align*}
We denote the integral kernel of $e^{itH}$ by $e^{itH}(r_1,y_1,r_2,y_2)$ for $r_1,r_2\in (0,\infty)$ and $y_1,y_2\in \rho\mathbb{S}^{n-1}$. Since $H$ commutes with the complex conjugation and $e^{itH}$ is unitary, we have
\begin{align*}
e^{-itH}(r_1,y_1,r_2,y_2)=\overline{e^{itH}(r_2,y_2,r_1,y_1)}.
\end{align*}
In the following, we study the properties of $e^{itH}$ only for $t>0$. The estimate for $t<0$ directly follows from this identity. The two-dimensional case has already been studied in \cite{F} and \cite{Z}; therefore, we focus on dimensions greater than two, although the method used in this paper can also handle the two-dimensional case.

\begin{thm}\label{thm:dispersive}
Suppose $n\geq 3$.

\noindent$(i)$ There exist $C>0$ and $\e_0>0$ such that
\begin{align*}
|e^{itH}(r_1,y_1,r_2,y_2)|\leq C\frac{1}{t^{\frac{n}{2}}}\left(1+\frac{2t}{r_1r_2}\right)^{-\n_0+\frac{n-2}{2}}
\end{align*}
for $t>0$, $r_1,r_2>0$ and 
\begin{align*}
\left\{\begin{aligned}
&\text{$\cos^{-1}(\rho^{-2}y_1\cdot y_2)\in [\e_0,\pi-\e_0]$ when $\rho>0$,}\\
&\text{$y_1,y_2\in\rho\mathbb{S}^{n-1}$ when $\rho\geq 1$},\\
&\text{$\cos^{-1}(\rho^{-2}y_1\cdot y_2)\in [0,\pi-\e_0]$ when $\rho>\frac{1}{2}$}
\end{aligned}\right.,
\end{align*}
where we take the range of $\cos^{-1}$ as $[0,\pi]$.

\noindent$(ii)$ Suppose $\rho^{-1}\notin 2\mathbb{N}$ and $\rho<1$. Then there exist $C>0$ such that
\begin{align*}
|e^{itH}(r_1,y_1,r_2,y_2)|\leq C\frac{1}{t^{\frac{n}{2}}}\left(1+\frac{r_1r_2}{2t}\right)^{\frac{n-2}{2}}\left(1+\frac{2t}{r_1r_2}\right)^{-\n_0+\frac{n-2}{2}}
\end{align*}
for $t>0$, $r_1,r_2>0$ and $y_1,y_2\in\rho\mathbb{S}^{n-1}$.

\end{thm}

\begin{rem}
The terms involving $\left(1+\frac{2t}{r_1r_2}\right)^{-\n_0+\frac{n-2}{2}}$ cannot be removed when $c<0$ (as is anticipated in \cite[after Remark 1.12]{FFFP2}). See Remark \ref{rem:lowoptimal}.
\end{rem}

\begin{rem}
We do not consider the cases $\rho^{-1}\notin 2\mathbb{N}$ in $(iv)$ for technical reasons.
\end{rem}

The next theorem shows the optimality of the estimate given in Theorem \ref{thm:dispersive} $(ii)$ on the region $\frac{r_1r_2}{2t}\gg 1$. Before the statement, we define 
\begin{align}\label{eq:phasecritrue}
\mathscr{D}_{\rho,\s_1}(\f_0):=&\{\m\in (0,1)\mid \cos^{-1}(\m)=\s_1\left(\frac{\pi}{2}+\rho \f_0+2\pi \rho q\right)   \quad \text{for some $q\in \mathbb{Z}$}\},
\end{align}
for $\f_0=0$ or $\f_0=\pi$. We note that $\mathscr{D}_{\rho,\s_1}(\f_0)$ is a finite set and $\mathscr{D}_{\rho,\s_1}(\f_0)=\emptyset$ when $\rho\geq 1$.

\begin{thm}\label{thm:optimality}
 Suppose $n\geq 3$ and $\rho^{-1}\notin 2\mathbb{N}$. Then,
\begin{align*}
e^{itH}(r_1,y_1,r_2,y_1)
=&\frac{e^{-\frac{r_1^2+r_2^2}{4it}}}{t^{\frac{n}{2}}}\cdot \left(\frac{r_1r_2}{2t}\right)^{\frac{n-2}{2}}\sum_{\s_1\in \{\pm\}}\sum_{\m_0\in \mathscr{D}_{\rho,\s_1}(0)} c_{n,\rho}(\m_0)e^{i\s_1\sqrt{1-\m_0^2}\frac{r_1r_2}{2t}}+O\left(\frac{1}{t^{\frac{n}{2}}}\left(\frac{r_1r_2}{2t} \right)^{\frac{n-3}{2}} \right)\\
e^{itH}(r_1,y_1,r_2,-y_1)
=&\frac{e^{-\frac{r_1^2+r_2^2}{4it}}}{t^{\frac{n}{2}}}\cdot \left(\frac{r_1r_2}{2t}\right)^{\frac{n-2}{2}}\sum_{\s_1\in \{\pm\}}\sum_{\m_0\in \mathscr{D}_{\rho,\s_1}(\pi)} c_{n,\rho}(\m_0)e^{i\s_1\sqrt{1-\m_0^2}\frac{r_1r_2}{2t}}+O\left(\frac{1}{t^{\frac{n}{2}}}\left(\frac{r_1r_2}{2t} \right)^{\frac{n-3}{2}} \right)
\end{align*}
for $t>0$, $r_1r_2/(2t)\gg 1$ and $y_1\in\rho\mathbb{S}^{n-1}$. Here, $c_{n,\rho}(\m_0)\in\mathbb{C}\setminus \{0\}$ and note that
\begin{align}\label{eq:Dnonempty}
\rho<\frac{1}{2}\Rightarrow \bigcup_{\s_1\in\{\pm\}}\mathscr{D}_{\rho,\s_1}(0)\neq \emptyset,\quad \rho<1\Rightarrow \bigcup_{\s_1\in\{\pm\}}\mathscr{D}_{\rho,\s_1}(\pi)\neq \emptyset.
\end{align}
In particular, there exist $C,C'>0$ such that
\begin{align*}
\rho<\frac{1}{2}\Rightarrow& |e^{itH}(r_1,y_1,r_2,y_1)|\geq C\frac{1}{t^{\frac{n}{2}}}\left(1+\frac{r_1r_2}{2t}\right)^{\frac{n-2}{2}}\\
\rho<1\Rightarrow& |e^{itH}(r_1,y_1,r_2,-y_1)|\geq C\frac{1}{t^{\frac{n}{2}}}\left(1+\frac{r_1r_2}{2t}\right)^{\frac{n-2}{2}}
\end{align*}
for $t>0$, $r_1r_2/(2t)\geq C'$ and $y_1\in\rho\mathbb{S}^{n-1}$

\end{thm}

This theorem shows that the dispersive estimate fails on the diagonal set $\{y_1=y_2\}$ when $\rho<1/2$ and at  conjugate pairs\footnote{On a Riemannian manifold $(N,h)$, we say that $(y_1,y_2)\in N\times N$ is a conjugate pair if there exists a non-trivial Jacobi field $J$ along a geodesic connecting $y_1$ and $y_2$ such that $J$ vanishes at the end points. It is known that the behavior of the wave kernel $\cos(t\sqrt{-\Delta_Y})(y_1,y_2)$ or the spectral projection becomes more ``singular" at a conjugate pair $(y_1,y_2)$ than at other points since a conjugate pair corresponds to a caustic of a Lagrangian submanifold generated by the geodesic flow. For the sphere $\rho\mathbb{S}^{n-1}$ with $n\geq 3$, $(y_1,y_2)$ is a conjugate pair if and only if $y_2=y_1$ or $y_2=-y_1$.} $(y_1,-y_1)$ when $\rho<1$ (under the additional assumption $\rho^{-1}\notin 2\mathbb{N}$).

We note that $c\geq 0$ implies $\n_0-\frac{n-2}{2}\geq 0$. As a corollary of the above two theorems, we have

\begin{cor}\label{Cor:main}
Suppose $n\geq 3$ and $c\geq 0$.

\noindent$(i)$ Suppose $\rho\geq 1$. Then there exist $C>0$ such that
\begin{align}\label{eq:dispest}
\|e^{itH}\|_{L^1\to L^{\infty}}\leq C\frac{1}{t^{\frac{n}{2}}}\quad (t>0)
\end{align}

\noindent$(ii)$ There exist $C>0$ and $\e_0>0$ such that
\begin{align*}
|e^{itH}(r_1,y_1,r_2,y_2)|\leq C\frac{1}{t^{\frac{n}{2}}}
\end{align*}
for $t>0$, $r_1,r_2>0$ and $y_1,y_2\in\rho\mathbb{S}^{n-1}$ satisfying $\cos^{-1}(\rho^{-2}y_1\cdot y_2)\in [\e_0,\pi-\e_0]$.

\noindent$(iii)$ Suppose $\rho>\frac{1}{2}$. Then there exist $C>0$ and $\e_0>0$ such that
\begin{align}\label{eq:disprho1/2}
|e^{itH}(r_1,y_1,r_2,y_2)|\leq C\frac{1}{t^{\frac{n}{2}}}
\end{align}
for $t>0$, $r_1,r_2>0$ and $y_1,y_2\in\rho\mathbb{S}^{n-1}$ satisfying $\cos^{-1}(\rho^{-2}y_1\cdot y_2)\in [0,\pi-\e_0]$.

\noindent$(iv)$ When both $\rho^{-1}\notin 2\mathbb{N}$ and $\rho<1$ hold, the dispersive estimate \eqref{eq:dispest} does not hold.

\end{cor}

\begin{rem}
The dispersive estimate is proved for $n=2$ in \cite{F} for all radii $\rho>0$. Higher dimensional cases with for $\rho>1$ are dealt with in \cite{JZ}. The results for $n=3$ and $\rho=1$ (the Euclidean cases) are proved in \cite{FFFP1, FFFP2}.

\end{rem}

The Strichartz estimates, which play important roles in the study of nonlinear Schr\"odinger equations are proved by the dispersive estimate thanks to the result by Keel-Tao \cite{KT}.
The third result in Corollary \ref{Cor:main} shows that the dispersive estimate holds near the diagonal set $\{y_1=y_2\}$ of the sphere when $\rho>\frac{1}{2}$. By using a partition of unity (like the argument in \cite[\S 6.2]{TT}) and the $TT^*$-argument in \cite{KT}, we recover the homogeneous Strichartz estimates for $\rho>\frac{1}{2}$, which were already proved in \cite{ZZ0}. On the other hand, Theorem \ref{thm:optimality} shows that the dispersive estimate does not hold even on the diagonal set $\{r_1=r_2,y_1=y_2\}$ when $\rho<\frac{1}{2}$ and $\rho^{-1}\notin 2\mathbb{N}$ while the Strichartz estimates still hold. In \cite[\S 3.2]{ZZ0}, the authors showed that the dispersive estimate holds if the propagator is ``microlocalized" on the phase space (see \cite[Proposition 6.1]{HZ} for the estimate on asymptotically conic spaces). This microlocalized dispersive estimate is sufficient for proving the Strichartz estimates.

\subsection{Expression of the integral kernel}
First, we rewrite the integral kernel of $e^{itH}$ in an abstract form.
For $t>0$, $r_1,r_2\in (0,\infty)$, and $y_1,y_2\in \rho\mathbb{S}^{n-1}\subset \re^n$,
\begin{align}\label{eq:propagatorsum}
e^{itH}(r_1,y_1,r_2,y_2)=\frac{c_n'}{i\rho^{n-1}}\cdot\frac{e^{-\frac{r_1^2+r_2^2}{4it}}}{(2t)^{\frac{n}{2}}}\mathscr{I}_{\rho,\frac{n-2}{2},c}\left(\frac{r_1r_2}{2t},\frac{1}{\rho^2}y_1\cdot y_2\right).
\end{align}
Here, $y_1\cdot y_2$ denote the inner product of the vector $y_1,y_2\in \rho\mathbb{S}^{n-1}\subset \re^n$, $c_n'\in \re\setminus \{0\}$ is a constant depending only on the dimension $n$ and for $\rho,d>0$ and $c$ satisfying \eqref{eq:csubcri}, we define
\begin{align}
\mathscr{I}_{\rho,d,c}(x,\f):&=x^{-d}\sum_{m=0}^{\infty}e^{-\frac{\pi}{2}i\n_m}J_{\n_m}(x)(m+d)d^{-1}C_m^d(\cos\f)\quad (x>0,\, \f\in[0,\pi]),\label{def:scrI}\\
\n_m:=&\sqrt{\frac{1}{\rho^2}m(m+2d)+d^2+c},\label{def:num}
\end{align}
where $J_{\n}(x)$ is the Bessel function and $C_{m}^d$ is the Gegenbauer polynomial.
When $d=\frac{n-2}{2}$, then $\n_m$ is the eigenvalue of the operator $-\Delta_{\rho\mathbb{S}^{n-1}}+\frac{(n-2)^2}{4}+c$. The expression \eqref{eq:propagatorsum} is deduced from Cheeger's functional calculus (\cite[$(2.10)$]{JZ}, see also \cite{C} and \cite[$(10)$]{F}) and the fact that the orthonormal projection on the sphere $\rho \mathbb{S}^{n-1}$ with radius $\rho$ associated with the eigenvalue $\rho^{-2}m(m+n-2)$ is given by
\begin{align*}
c_n\rho^{-n+1}(m+d)C_{m}^d(\rho^{-2}y_1\cdot y_2),\quad (y_1,y_2\in\rho\mathbb{S}^{n-1})  \quad d:=\frac{n-2}{2}.
\end{align*}
In the following of the paper, we study the properties of $\mathscr{I}_{\rho,d,c}$ and establish uniform estimates with respect to $x$ and $\f$ for fixed constants $\rho,d>0$ and $c$ satisfying \eqref{eq:csubcri}.

\subsection*{Acknowledgment}
The author is grateful to Tomomi Yokota for encouraging him to write about the results on the dispersive estimates in the presence of singular potentials as an application of the method used in his previous joint paper \cite{TT} with Tamori. He would like to thank Junyong Zhang for his interest in this subject and for pointing out some mistakes in the previous draft. He was supported by JSPS KAKENHI Grant Number 23K13004.

\section{Asymptotic behavior of the sum for small $x$}

\begin{prop}\label{prop:smallx}
Then for each $R>0$, there exists $C>0$ such that
\begin{align*}
|\mathscr{I}_{\rho,d,c}(x,\f)|\leq C(1+|x|^{-1})^{-\n_0+d}
\end{align*}
for $|x|\leq R$ and $\f\in [0,\pi]$.

\end{prop}

\begin{proof}
The proof is totally similar to that in \cite[Proposition 4.1]{F}, \cite[Proposition 4.1]{JZ}, or \cite[\S 6.1]{TT}. Using the universal bound (the second estimate below is given in \eqref{eq:Gegenunif})
\begin{align*}
|J_{\n_m}(x)|\leq \frac{1}{\Gamma(\n_m+1)}\left(\frac{|x|}{2} \right)^{\n_m},\quad |C_m^d(\cos\f)|\leq Cm^{2\n-1},
\end{align*}
we obtain
\begin{align*}
|\mathscr{I}_{\rho,d,c}(x,\f)|\lesssim \sum_{m=0}^{\infty}\frac{(1+|x|^{-1})^{-\n_m+d}}{\Gamma(\n_m+1)2^{\n_m}}m^{2d}\lesssim (1+|x|^{-1})^{-\n_0+d}\sum_{m=0}^{\infty}\frac{1}{\Gamma(\n_m+1)2^{\n_m}}m^{2d}
\end{align*}
for $|x|\leq R$ and $\f\in [0,\pi]$. Here, the sum is convergent due to the Stirling formula with $\n_m=\rho^{-1}m+O(1)$ as $m\to \infty$. Thus, the proposition follows.

\end{proof}

\begin{rem}\label{rem:lowoptimal}
We can also prove that $|\mathscr{I}_{\rho,d,c}(x,\f)|\geq C'(1+|x|^{-1})^{-\n_0+d}$ for $|x|\leq R$ and $\f\in [0,\pi]$ by using the elementary bound  $|J_{\n_m}(x)|\gtrsim |x|^{\n_m}$ for small $x$. This implies that the term involving $\left(1+\frac{2t}{r_1r_2}\right)^{-\n_0+\frac{n-2}{2}}$ in Theorem \ref{thm:dispersive} cannot be removed unless $c\geq 0$. In other words, the usual dispersive estimate \eqref{eq:dispest} does not hold for $c<0$.
\end{rem}

\section{Asymptotic behavior of the sum for large $x$}

\subsection{Decomposition of the sum}\label{subsec:decom}

From Proposition \ref{prop:smallx}, it suffices to study the asymptotic behavior of $\mathscr{I}_{\rho,d,c}(x,\f)$ for $x\gg 1$. To do this, we use the strategy developed in \cite{TT}. The main difference from \cite{TT} is the derivation of Proposition \ref{prop:osciint}, which clarifies the condition under which we obtain improved decays (\cite[Proposition 5.2]{TT}) of the sum and provides a more precise asymptotic expansion of the sum $\sI_{\rho,d}$.

Corresponding to the asymptotic behavior of the Bessel functions, we define
\begin{align*}
&\Omega_1=\{m\in[1,\infty)\mid 1\leq \n_m\leq x-\frac{1}{2}x^{\frac{1}{3}}\},\,\, \Omega_2=\{m\in[1,\infty)\mid  x-2x^{\frac{1}{3}}\leq \n_m\leq x+2x^{\frac{1}{3}}\},\\
&\Omega_3=\{m\in[1,\infty)\mid  \n_m\geq x+\frac{1}{2}x^{\frac{1}{3}}\}.
\end{align*}

\begin{lem}\cite[Lemma 3.5]{TT}\label{Cutoff}
For $x\gg 1$, 
there exist $\chi_0\in C_c^{\infty}(\re;[0,1])$ and $\chi_{j,x}\in C^{\infty}(\re;[0,1])$  ($1\le j\le 3$) such that $\chi_0(m)=1$ for $\n_m\leq 2$ or $m\leq 1$ and
\begin{align}
&\chi_0(m)+\sum_{j=1}^3\chi_{j,x}(m)=1\ \text{ for $m\geq 0$},\quad \supp \chi_{j,x}\subset \Omega_j,\,\, |\pa_m^{\a}\chi_{2,x}(m)|\leq C_{\a}x^{-\frac{\a}{3}}\label{POUproper}\\
&|\pa_m^{\a}\chi_{1,x}(m)|\leq C_{\a}\max(m^{-\a}, (x-\n_m)^{-\a}),\,\, |\pa_m^{\a}\chi_{3,x}(m)|\leq C_{\a}(\n_m-x)^{-\a}\nonumber.
\end{align}
\end{lem}

For $\s=(\s_1,\s_2)\in \{\pm\}\times \{\pm\}$ and $j=2,3$, we set
\begin{align*}
&S_{1,\s}(m,x,\f)=\s_1xh_1\left(\frac{m}{\rho x}\right) +\s_2\f m-\frac{\pi}{2\rho}m,\quad S_{\s_2}(m,\f)=\s_2\f m-\frac{\pi}{2\rho}m,
\end{align*}
where $h_1$ is defined by
\begin{align*}
h_1(\m):=\sqrt{1-\m^2}-\m\cos^{-1}(\m),
\end{align*}
where we take the range of $\cos^{-1}(\m)$ ($\m\in [0,1]$) as $[0,\frac{\pi}{2}]$.
We take $x\gg 1$ such that $\Omega_2\cup \Omega_3\subset \{\n_m\geq 8\}$ in order to apply Proposition \ref{Besselasymp}. 
We define
\begin{align*}
&\z_{1,\s}(m,x,\f)=x^{-d-\frac{1}{4}}\chi_{1,x}(m)(m+d)(x-\n_m)^{-\frac{1}{4}}g_{d,\s_2}(m,\f)a_{\s_1,x}(\n_m)e^{i\s_1x\left(h_1\left(\frac{\n_m}{x}\right)-h_1\left(\frac{m}{x}\right) \right)-\frac{\pi}{2}i(\n_m-m) },\\
&\z_{j,\s_2}(m,x,\f)=x^{-d}\chi_{j,x}(m)(m+d)J_{\n_m}(x)g_{d,\s_2}(m,\f)e^{-\frac{\pi}{2}i(\n_m-m)},\quad j=2,3,
\end{align*}
where $a_{\pm,x}$ and $g_{d,\pm}$ are defined in Propositions \ref{Besselasymp} and \ref{Gegenasymp} respectively.

Using Propositions \ref{Besselasymp}, \ref{Gegenasymp}, the definition \eqref{def:scrI} of $\mathscr{I}_{\rho,d,c}$, and Lemma \ref{Cutoff}, we can write
\begin{align}\label{sIdecom}
\sI_{\rho,d,c}(x,\f)=\sum_{\s\in \{\pm\}\times \{\pm\}}I_{1,\s}(x,\f)+\sum_{j=2}^3\sum_{\s_2\in \{\pm\}}I_{j,\s_2}(x,\f)+R(x,\f),
\end{align}
where we set
\begin{align*}
I_{1,\s}(x,\f):=&\sum_{m=1}^{\infty}\z_{1,\s}(m,x,\f)e^{iS_{1,\s}(m,x,\f)},\,\, I_{j,\s_2}(y,\f):=\sum_{m=1}^{\infty}\z_{j,\s_2}(m,x,\f)e^{iS_{\s_2}(m,\f)},\\
R(x,\f):=&x^{-d}\sum_{m=0}^{\infty}\chi_0(m)e^{-\frac{\pi}{2}i\n_m}J_{\n_m}(x)(m+d)d^{-1}C^{d}_m(\cos\f)\\
&+x^{-d}\sum_{m=0}^{\infty}
(1-\chi_0(m))
e^{-\frac{\pi}{2}i\n_m}J_{\n_m}(x)(m+\nu)r(m,\f) 
\end{align*}
for $j=2,3$.
The remainder term $R(x,\f)$ is easy to handle:

\begin{lem}\label{lem:Rest}
We have
\begin{align*}
|R(x,\f)|\lesssim x^{-d-\frac{1}{3}}
\end{align*}
for $x\gg 1$ and $\f\in [0,\pi]$.

\end{lem}

\begin{proof}
We recall $J_{\n_m}(x)=O(x^{-\frac{1}{3}})$, $d^{-1}C_m^d(\cos \f)=O((1+m)^{2d-1})$, and $r(m,\f)=O((1+m)^{-N})$ for all $N>0$ from \eqref{eq:Besselunif}, \eqref{eq:Gegenunif} and Proposition \ref{Gegenasymp}. Since $\chi_0$ is compactly supported,
\begin{align*}
|R(x,\f)|\lesssim x^{-d}\sum_{m=0}^{\infty}\left(\chi_0(m)\cdot x^{-\frac{1}{3}}\cdot m\cdot (1+m)^{2d-1}+x^{-\frac{1}{3}}\cdot m\cdot  (1+m)^{-N} \right)\lesssim x^{-d-\frac{1}{3}}
\end{align*}
if we take $N>2$.

\end{proof}

It is relatively easy to estimate $I_{j,\s_2}$ (for $j=2,3$) since the phase function $S_{\s_2}$ of $I_{j,\s_2}$ ($j=2,3$) is simple. The proof of the next proposition is similar to that in \cite[Propositions 4.1,4.2]{TT} and given in Appendix \ref{subsec:I_23}.

\begin{prop}\label{prop:I_23}
Let $j=2,3$ and $\s_2\in \{\pm\}$. 

\noindent$(i)$ For each $0<\e<\pi$, we have $|I_{j,\s_2}(x,\f)|\lesssim 1$ for $x\gg 1$ and $\f\in [\e,\pi-\e]$.

\noindent$(ii)$ If $\rho^{-1}\notin 2\mathbb{N}$, then we can find $\e>0$ such that for each $N>0$, we have $|I_{j,\s_2}(x,\f)|\lesssim x^{-N}$ for $x\gg 1$ and $\f\in [0,\e]\cup [\pi-\e,\pi]$.

\noindent$(iii)$
\begin{align*}
|I_{j,\s_2}(x,\f)| \lesssim  \left\{\begin{aligned}
&1\quad &&\text{when}\quad \rho^{-1}\notin 2\mathbb{N}  \\
&x^{d}\quad &&\text{when}\quad \rho^{-1}\in 2\mathbb{N}
\end{aligned}\right.
\end{align*} 
for $x\gg 1$ and $\f\in [0,\pi]$.

\end{prop}

\begin{rem}
The condition $\rho^{-1}\notin 2\mathbb{N}$ comes from the condition $\pa_mS_{\s_2}(m,x,\f)|_{\f=0\,\text{or}\,\pi}\notin 2\pi\mathbb{Z}$.
\end{rem}

\subsection{Properties of the phase function and the amplitude}

In the following, we study the sum $I_{1,\s}$ defined in Subsection \ref{subsec:decom}. From \cite[$(2.1)$]{TT}, we can write $I_{1,\s}$ as
\begin{align}\label{eq:I_1sum}
I_{1,\s}(x,\f)=\sum_{q\in\mathbb{Z}}\int_{\re} \z_{1,\s,q}(m,x,\f)e^{i(S_{1,\s}(m,x,\f)+2\pi qm)}dm,
\end{align}
where we define
\begin{align}\label{eq:zetaqdef}
\z_{1,\s,0}=\z_{1,\s}\quad \z_{1,\s,q}=-\frac{1}{2\pi i q}\left(\pa_m\z_{1,\s}+i(\pa_mS_{1,\s})\z_{1,\s}\right)\quad (q\in\mathbb{Z}\setminus \{0\}).
\end{align}

\subsection*{Critical points of the phase function}

Now we study the location of the critical points of the phase function  $S_{1,\s}(m,x,\f)+2\pi qm$ in \eqref{eq:I_1sum}.
From the definition of $S_{1,\s}$ given in Subsection \ref{subsec:decom}, one has
\begin{align}\label{eq:S_1cal}
\pa_mS_{1,\s}(m,x,\f)=-\frac{\s_1}{\rho}\cos^{-1}\left(\frac{m}{\rho x}\right)+\left(\s_2\f-\frac{\pi}{2\rho}\right).
\end{align}
For $\s\in \{\pm\}\times \{\pm\}$, $\f\in [0,\pi]$, we define
\begin{align}
\mathscr{C}_{\s,\rho}(\f):=&\bigcup_{q\in\mathbb{Z}}\mathscr{C}_{\s,\rho}(\f,q)\label{eq:defcri1}\\
\mathscr{C}_{\s,\rho}(\f,q):=&\left\{\m\in [0,1]\mid-\frac{\s_1}{\rho}\cos^{-1}\left(\m\right)+\left(\s_2\f-\frac{\pi}{2\rho}\right)+2\pi q=0 \right\}.\label{eq:defcri2}
\end{align}
Clearly, 
\begin{align}\label{eq:S_1phasecha}
\pa_m(S_{1,\s}(m,x,\f)+2\pi qm)=0 \Leftrightarrow \frac{m}{\rho x}\in \mathscr{C}_{\s}(\f,q)
\end{align}
and $\pa_mS_{1,\s}(m,x,\f)\in 2\pi\mathbb{Z} \Leftrightarrow \frac{m}{\rho x}\in \mathscr{C}_{\s}(\f)$.

\begin{lem}\label{lem:phasecript}
Let $\rho>0$ and $\s=(\s_1,\s_2)\in \{\pm\}\times \{\pm\}$.
 
\noindent$(i)$ We have $\sup_{\f\in [0,\pi]}\#\mathscr{C}_{\s,\rho}(\f)<\infty$.

\noindent$(ii)$ If $\rho>1$, we have $\mathscr{C}_{\s,\rho}(\f,q)=\emptyset$ for $(\f,q)\in (0,\pi)\times \mathbb{Z}\setminus \{0\}$. Moreover,
\begin{align*}
\f\in \{0,\pi\}\Rightarrow \mathscr{C}_{\s,\rho}(\f,q)=\left\{\begin{aligned}
&\{0\} &&\text{when $\s\in \{-\}\times \{\pm\}$ and $(\f,q)=(0,0)$}\\
&\emptyset &&\text{otherwise.}
\end{aligned}\right.
\end{align*}

\noindent$(iii)$ If $\rho=1$, we have $\mathscr{C}_{\s,\rho}(\f,q)=\emptyset$ for $(\f,q)\in (0,\pi)\times \mathbb{Z}\setminus \{0\}$. Moreover,
\begin{align*}
\f\in \{0,\pi\}\Rightarrow \mathscr{C}_{\s,\rho}(\f,q)=\left\{\begin{aligned}
&\{0\} &&\text{when $\s\in \{-\}\times \{\pm\}$ and $(\f,q)=(0,0)$}\\
&\{0\} &&\text{when $\s\in \{+\}\times \{+\}$ and $(\f,q)=(\pi,0)$}\\
&\{0\} &&\text{when $\s\in \{+\}\times \{-\}$ and $(\f,q)=(\pi,1)$}\\
&\emptyset &&\text{otherwise.}
\end{aligned}\right.\\
\end{align*}

\noindent$(iv)$ If $\rho>\frac{1}{2}$, we have  $\mathscr{C}_{\s,\rho}(\f,q)=\emptyset$ for $(\f,q)\in \left(0,\frac{\pi}{2}\right)\times \mathbb{Z}\setminus \{0\}$. Moreover,
\begin{align*}
\f=0\Rightarrow \mathscr{C}_{\s,\rho}(\f,q)=\left\{\begin{aligned}
&\{0\} &&\text{when $\s\in \{-\}\times \{\pm\}$ and $(\f,q)=(0,0)$}\\
&\emptyset &&\text{otherwise.}
\end{aligned}\right.
\end{align*}

\noindent$(v)$ If $\rho^{-1}\notin 2\mathbb{N}$, then $1\notin \mathscr{C}_{\s,\rho}(0)\cup \mathscr{C}_{\s,\rho}(\pi)$.

\noindent$(vi)$ We have 
\begin{align*}
\mathscr{C}_{(\s_1,+),\rho}(0,q)=\mathscr{C}_{(\s_1,-),\rho}(0,q),\quad \mathscr{C}_{(\s_1,+),\rho}(\pi,q)= \mathscr{C}_{(\s_1,-),\rho}(\pi,q+1).
\end{align*}
More precisely, for $\f_0=0$ or $\f_0=\pi$, one has $\mathscr{D}_{\rho,\s_1}(\f_0)=\bigsqcup_{q\in\mathbb{Z}}\mathscr{C}_{\s,\rho}(\f_0,q)\cap (0,1)$ and
\begin{align*}
\mathscr{C}_{(\s_1,+),\rho}(\f_0,q)\cap \mathscr{C}_{(\s_2,-),\rho}(\f_0,q')=\left\{\begin{aligned}
&\mathscr{C}_{(\s_1,+),\rho}(\f_0,q) &&\text{$q'=q$ and $\f_0=0$}\\
&\mathscr{C}_{(\s_1,+),\rho}(\f_0,q) &&\text{$q'=q+1$ and $\f_0=\pi$}\\
&\emptyset &&\text{otherwise},
\end{aligned}\right.
\end{align*}
where $\mathscr{D}_{\rho,\s_1}(\f_0)$ is defined in \eqref{eq:phasecritrue}.

\end{lem}

\begin{proof}
The proof is elementary. We omit the detail.
\end{proof}

\subsection*{Symbolic estimates}

In order to apply the stationary phase method, we also need symbolic estimates for the phase function $S_{1,\s}$ and the amplitude $\z_{1,\s,q}$ appearing in \eqref{eq:I_1sum}.

\begin{lem}\label{lem:phase}
Let $0<\e_1<1$ and $\s\in \{\pm\}\times \{\pm\}$.
For $\a\geq 1$, there exists $C_{\a}>0$ such that
\begin{align*}
|\pa_m^{\a+1}S_{1,\s}(m,x,\f)|\leq \begin{cases}C_{\a}x^{-\a}\quad \text{for}\quad 1\leq \n_m\leq \e_1 x\\
C_{\a}x^{-\frac{1}{2}}(x-\n_m)^{-\a+\frac{1}{2}}\quad \text{for}\quad \e_1 x\leq \n_m\leq x-\frac{1}{2}x^{\frac{1}{3}}
\end{cases}
\end{align*}
and for $m\in \Omega_1$, $x\gg 1$ and $\f\in [0,\pi]$.

\end{lem}

\begin{lem}\label{lem:amp}

Let $\s=(\s_1,\s_2)\in \{\pm\}\times\{\pm\}$, $N>0$, $0<\e<\pi/2$ and $\a\in\mathbb{N}$. 

\noindent$(i)$ $\supp \z_{1,\s,q}(\cdot,x,\f)\subset \Omega_1$.

\noindent$(ii)$ Let $0<\e_1<1$ and $0<\e_0<\frac{\pi}{2}$. Then, we have
\begin{align*}
|\pa_{m}^{\a}\z_{1,\s,q}(m,x,\f)|\lesssim \begin{cases}x^{-d-\frac{1}{2}}(1+m)^{2d-\a}(1+m\sin \f)^{-d} \quad \text{when}\quad \n_m\leq \e_1 x \\
x^{d-\frac{1}{4}} (x-\m_m)^{-\frac{1}{4}-\a}(1+m\sin \f)^{-d}\quad\text{when}\quad \e_1 x\leq \n_m\leq x-\frac{1}{2}x^{\frac{1}{3}}
\end{cases}
\end{align*}
for $m\geq 1$, $x\gg 1$, $\f\in [0,\pi]$ and $q\in\mathbb{Z}$. In particular, we have
\begin{align*}
|\pa_{m}^{\a}\z_{1,\s,q}(m,y,\f)|\lesssim \begin{cases}x^{-d-\frac{1}{2}}(1+m)^{d-\a}\quad \text{when}\quad \n_m\leq \e_1 x \\
x^{-\frac{1}{4}} (x-\n_m)^{-\frac{1}{4}-\a}\quad\text{when}\quad \e_1 x\leq \n_m\leq x-\frac{1}{2}x^{\frac{1}{3}}
\end{cases}
\end{align*}
for $m\geq 1$, $x\gg 1$, $\f\in [\e_0,\pi-\e_0]$, and $q\in\mathbb{Z}$.

\noindent$(iii)$ Suppose that $\f_0=0$ or $\f_0=\pi$. Let $q_0\in\mathbb{Z}$ and $\m_0\in \mathscr{C}_{\s,\rho}(\f_0,q_0)\cap (0,1)$. Then we have
\begin{align*}
\z_{1,\s,q_0}(\rho x\m_0,x,\f_0)=\frac{\rho\m_0}{(2\pi)^{\frac{1}{2}}(1-\m_0^2)^{\frac{1}{4}}}e^{i\left(L_{\s_1,d,\rho,\m_0}-\s_1\frac{\pi}{4} \right)} \times x^{-d+\frac{1}{2}} g_{d,\s_2}(\rho\m_0x,\f_0)    +O(x^{d-1})
\end{align*}
for $x\gg 1$, where $g_{d,\s_2}$ is the function given in Proposition \ref{Gegenasymp} and $L_{\s_1,d,\rho,\m_0}\in \re$ is a constant depending only on $\s_1,d,\rho,\m_0$.

\end{lem}

\begin{rem}
The asymptotic behavior of the function $g_{d,\s_2}$, which arises from the expansion of the Gegenbauer polynomials cannot be determine in a canonical way. However, the asymptotic of $g_{d,+}+g_{d,-}$ can be explicitly calculated as we will see in Proposition \ref{Gegenasymp}.

\end{rem}

\begin{proof}
These lemmas directly follow from the definition of $S_{1,\s}$ and $\z_{1,\s}$ given in Subsection \ref{subsec:decom}, and Propositions \ref{Besselasymp}, \ref{Gegenasymp}. Here, we prove Lemma \ref{lem:amp} $(iii)$ only.

We see from the definition \eqref{def:num} of $\n_m$ that
\begin{align}\label{eq:numasym}
\n_m=\sqrt{\rho^{-2}(m+d)^2+d^2+c}=\rho^{-1}(m+d)+O(m^{-1})\quad \text{as}\quad (m\to \infty).
\end{align}
Since $\m_0\in (0,1)$, there exists $0<\e_1<1$ such that $\n_{\rho x\m_0}\leq \e_1x$ for sufficiently large $x$. Then, Lemma \ref{lem:amp} $(ii)$ implies $(\pa_m\z_{1,\s})(\rho x\m_0,x,\f_0)=O(x^{d-\frac{3}{2}})$ and the assumption $\m_0\in \mathscr{C}_{\s,\rho}(\f_0,q_0)$ means $(\pa_mS_{1,\s})(\rho x\m_0,x,\f_0)=-2\pi q_0$. Therefore, 
\begin{align}\label{eq:crizetaq_0}
\z_{1,\s,q_0}(\rho x\m_0,x,\f_0)=\z_{1,\s}(\rho x\m_0,x,\f_0)+O(x^{d-\frac{3}{2}})
\end{align}
for $x\gg 1$ by \eqref{eq:zetaqdef}.

Now we study the asymptotic behavior of $\z_{1,\s}$. We recall that $\z_{1,\s}$ is the function of the form
\begin{align*}
x^{-d-\frac{1}{4}}\chi_{1,x}(m)(m+d)(x-\n_m)^{-\frac{1}{4}}g_{d,\s_2}(m,\f)a_{\s_1,x}(\n_m)e^{i\s_1x\left(h_1\left(\frac{\n_m}{x}\right)-h_1\left(\frac{m}{x}\right) \right)-\frac{\pi}{2}i(\n_m-m) }.
\end{align*}
It follows from Propositions \ref{Besselasymp}, \ref{Gegenasymp}, and \eqref{eq:numasym} that
\begin{align*}
&x^{-d-\frac{1}{4}}\chi_{1,x}(m)(\rho x\m_0+d)(x-\n_{\rho x\m_0})^{-\frac{1}{4}}=\rho \m_0(1-\m_0)^{-\frac{1}{4}}x^{-d+\frac{1}{2}}+O(x^{-d-1})\\
&h_1\left(\frac{\n_{\rho x\m_0}}{x}\right)-h_1\left(\m_0\right)=h_1\left(\m_0+\frac{d}{\rho x}+O\left(\frac{1}{x^2}\right)\right)-h_1\left(\m_0\right)=\frac{d}{\rho }h_1'(\m_0)x^{-1}+O(x^{-2})\\
&a_{\s_1,x}(\n_{\rho x\m_0})=\frac{e^{-\frac{\pi}{4}i\s_1}}{(2\pi)^{\frac{1}{2}}(1+\m_0)^{\frac{1}{4}}  }+O(x^{-\frac{1}{2}}),\quad \n_{\rho x\m_0}-\rho x\m_0=\rho^{-1}d+O(x^{-1}).
\end{align*}
Now, Lemma \ref{lem:amp} $(iii)$ is proved by these identities and \eqref{eq:crizetaq_0} by setting 
\begin{align*}
L_{\s_1,d,\rho,\m_0}=\s_1\frac{d}{\rho}h_1'(\m_0)-\frac{\pi d}{2\rho}=-\s_1\frac{d}{\rho}\cos^{-1}(\m_0)-\frac{\pi d}{2\rho}.
\end{align*}
\end{proof}

\subsection{Extraction of principal terms via the non-stationary phase method}\label{subsec:nonsta}

Let $\chi\in C^{\infty}(\re;[0,1])$ such that $\chi(m)=1$ for $m\leq 1$ and $\chi(m)=0$ for $m\geq 2$ and setting $\overline{\chi}=1-\chi$. We define $\chi_{\e,\m_0}(m)=\chi((\m-\m_0)/\e)$, where we write $m=\rho x\m$. We define $\overline{\chi_{\e,x}}=1-\chi_{\e,x}$ and
\begin{align*}
I_{\s,q,\m_0}^{\e}(x,\f):=\int_{\re}\chi_{\e,\m_0}(m) \z_{1,\s,q}(m,x,\f)e^{i(S_{1,\s}(m,x,\f)+2\pi qm)}dm.
\end{align*}

\begin{prop}\label{prop:Non-stationary}
Let $\s\in \{\pm\}\times \{\pm\}$ and $N>0$. For $\e>0$ sufficiently small, we have
\begin{align}\label{eq:Non-stationary}
I_{1,\s}(x,\f)=\sum_{q\in\mathbb{Z}}\sum_{\m_0\in \mathscr{C}_{\s,\rho}(\f,q)}I_{\s,q,\m_0}^{\e}(x,\f)+O(x^{-N})
\end{align}
for $x\gg 1$ and $\f\in [0,\pi]$.

\end{prop}

\begin{rem}
By Lemma \ref{lem:phasecript} $(i)$, the number of the sum in the right hand side is finite.
\end{rem}

To prove this proposition, we need the following lemma.

\begin{lem}\label{lem:nonstphase}
For $0\leq m\leq \rho x$ and $x\gg 1$, we write
\begin{align*}
\m:=\frac{m}{\rho x}\in [0,1].
\end{align*}

\noindent$(i)$ Let $q_0\in \mathbb{Z}$. Then, there exists $c>0$ such that if we choose $\e>0$ sufficiently small, 
\begin{align*}
|\pa_m(S_{1,\s}(m,x,\f)+2\pi q_0m )|\geq c
\end{align*}
for $\m\in [0,1]$, $x\gg 1$, and $\f\in [0,\pi]$ satisfying $\inf_{\m_0\in \mathscr{C}_{\s,\rho}(\f,q_0)}|\m-\m_0|\geq 2\e$

\noindent$(ii)$ There exist $c>0$ and $Q>0$ such that
\begin{align*}
|\pa_m(S_{1,\s}(m,x,\f)+2\pi qm )|\geq c(1+|q|)
\end{align*}
for $0\leq \m\leq 1$, $x\gg 1$, $\f\in [0,\pi]$ and $q\in\mathbb{Z}$ satisfying $|q|\geq Q$. In particular, $\mathscr{C}_{\s,\rho}(\f,q)=\emptyset$ for $|q|\geq Q$.

\end{lem}

\begin{proof}
The proof is elementary (just use \eqref{eq:S_1cal}, \eqref{eq:S_1phasecha}, and the definition of $\mathscr{C}_{\s,\rho}(\f,q_0)$). We omit the detail.

\end{proof}

\begin{proof}[Proof of Proposition \ref{prop:Non-stationary}]
Let $\s\in \{\pm\}\times \{\pm\}$ and $N>0$.
By \eqref{eq:I_1sum}, we write
\begin{align*}
I_{1,\s}(x,\f)=&\sum_{q\in\mathbb{Z}}\int_{\re} \z_{1,\s,q}(m,x,\f)e^{i(S_{1,\s}(m,x,\f)+2\pi qm)}dm\\
=&\sum_{q\in\mathbb{Z},|q|< Q}\sum_{\m_0\in \mathscr{C}_{\s,\rho}(\f,q)}I_{\s,q,\m_0}^{\e}(x,\f)\\
&+\sum_{q\in\mathbb{Z},|q|< Q}\int_{\re}\left(1-\sum_{\m_0\in \mathscr{C}_{\s,\rho}(\f,q)}\chi_{\e,\m_0}(m)\right) \z_{1,\s,q}(m,x,\f)e^{i(S_{1,\s}(m,x,\f)+2\pi qm)}dm\\
&+\sum_{q\in\mathbb{Z},|q|\geq Q}\int_{\re} \z_{1,\s,q}(m,x,\f)e^{i(S_{1,\s}(m,x,\f)+2\pi qm)}dm
\end{align*}
Here, the last term can be estimated as $O(x^{-N})$. In fact, by Lemmas \ref{lem:phase}, \ref{lem:amp} $(ii)$, and \ref{lem:nonstphase} $(ii)$, and the integration by parts, we have
\begin{align*}
\left|\int_{\re} \z_{1,\s,q}(m,x,\f)e^{i(S_{1,\s}(m,x,\f)+2\pi qm)}dm\right|\lesssim (1+|q|)^{-N-1}x^{-N}
\end{align*}
for $x\gg 1$, $\f\in [0,\pi]$ and $q\in\mathbb{Z}$ satisfying $|q|\geq Q$. This implies that the term involving $\sum_{q\in \mathbb{Z},|q|\geq Q}$ is $O(x^{-N})$.

To estimate the second term, we observe
\begin{align*}
\supp \left(1-\sum_{\m_0\in \mathscr{C}_{\s,\rho}(\f,q)}\chi_{\e,\m_0}(m)\right)\subset \bigcap_{\m_0\in \mathscr{C}_{\s,\rho}(\f,q)}\left\{m\mid \left|\frac{m}{\rho x}-\m_0\right|\geq 2\e\right\}
\end{align*}
by the choice of $\chi_{\e,\m_0}$. By using the integration by parts with Lemma \ref{lem:nonstphase} $(i)$, we conclude that the second term is $O(x^{-N})$. Finally, we obtain
\begin{align*}
I_{1,\s}(x,\f)=&\sum_{q\in\mathbb{Z},|q|< Q}\sum_{\m_0\in \mathscr{C}_{\s,\rho}(\f,q)}I_{\s,q,\m_0}^{\e}(x,\f)+O(x^{-N}).
\end{align*}
Since $\mathscr{C}_{\s,\rho}(\f,q)=\emptyset$ for $|q|\geq Q$ by Lemma \ref{lem:nonstphase} $(i)$, this proves \eqref{eq:Non-stationary}.

\end{proof}

\subsection{Estimates for principal terms via the stationary phase method}

In this subsection, we study the asymptotic behavior of the principal terms $I_{\s,q,\m_0}^{\e}$ in \eqref{eq:Non-stationary}.

\begin{prop}\label{prop:awayfromconj}
Let $\s=(\s_1,\s_2)\in \{\pm\}\times \{\pm\}$, $q_0\in\mathbb{Z}$, and $0<\e_0<\pi$. For sufficiently small $\e>0$, we have
\begin{align*}
|I_{\s,q_0,\m_0}^{\e}(x,\f)|\lesssim 1
\end{align*}
for $x\gg 1$, $\f\in [\e_0,\pi-\e_0]$, and $\m_0\in \mathscr{C}_{\s,\rho}(\f,q_0)$.

\end{prop}

The proof of this proposition is almost identical to that of \cite[Propositions 5.1]{TT} and is given in Appendix \ref{subsec:I_1awayconj}. The point here is to take an advantage of the improved decay of the amplitudes $\z_{1,\s,q}$ away from $\f=0,\pi$. We also point out that the angles $\f=0,\pi$ correspond to the diagonal set $\{y_1=y_2\}$ and the set of conjugate pairs $\{y_1=-y_2\}$ in the original problem (see \eqref{eq:propagatorsum}).

Next, we focus on studying the behavior of $I_{1,\s}$ near $\f=0,\pi$. 

\begin{prop}\label{prop:osciint}
Let $\s=(\s_1,\s_2)\in \{\pm\}\times \{\pm\}$ and $q_0\in\mathbb{Z}$. Suppose that $\f_0=0$ or $\f_0=\pi$. Then, the following statements hold true for sufficiently small $\e,\e_0>0$:

\noindent$(i)$ Suppose that $0\in \mathscr{C}_{\s,\rho}(\f_0,q_0)$. Then,
\begin{align*}
|I_{\s,q_0,\m_0}^{\e}(x,\f)|\lesssim 1
\end{align*}
for $x\gg 1$, $\f\in [\f_0-\e_0,\f_0+\e_0]\cap [0,\pi]$ and $\m_0\in \mathscr{C}_{\s,\rho}(\f,q_0)$.

\noindent$(ii)$ Suppose that $\mathscr{C}_{\s,\rho}(\f_0,q_0)\cap (0,1)\neq \emptyset$. Then,
\begin{align}\label{eq:baddecay}
|I_{\s,q_0,\m_0}^{\e}(x,\f)|\lesssim x^{d}
\end{align}
for $x\gg 1$, $\f\in [\f_0-\e_0,\f_0+\e_0]\cap [0,\pi]$ and $\m_0\in \mathscr{C}_{\s,\rho}(\f,q_0)$. Moreover,
\begin{align}\label{eq:I_1eachasym}
I_{\s,q_0,\m_0}(x,\f_0)=(\rho^2\m_0)e^{i\s_1\sqrt{1-\m_0^2}x+iL_{\s_1,\rho,d,\m_0}}x^{-d+1}g_{d,\s_2}(\rho\m_0x,\f_0) +O(x^{d-\frac{1}{2}})
\end{align}
for $x\gg 1$, where $L_{\s_1,\rho,d,\m_0}\in \re$ is the same constant as in Lemma \ref{lem:amp} $(iii)$. 

\end{prop}

\begin{rem}
In this paper, we de not handle the case $1\in \mathscr{C}_{\s,\rho}(\f_0,q_0)$ with $\f_0=0,\pi$. This does not happen when we assume $\rho^{-1}\notin 2\mathbb{N}$, thanks to Lemma \ref{lem:phasecript} $(v)$.

\end{rem}

\begin{proof}

By the change of variable $m=\rho x\m$, this integral is written as
\begin{align}\label{eq:I_1scale}
I_{\s,q_0,\m_0}^{\e}(x,\f)=x^{d+\frac{1}{2}}\int_{\re}e^{ixS(\m,\f) }\c_{x,\f}(\m) d\m,
\end{align}
where we set
\begin{align*}
S(\m,\f)=\s_1h_1(\m) +\left(\s_2\rho\f-\frac{\pi}{2}+2\pi\rho q_0 \right)\m  ,\quad \c_{x,\f}(\m)=\rho x^{-d+\frac{1}{2}}\chi_{\e,\m_0}(\rho x\m) \z_{1,\s,q_0}\left(\rho x\m,x,\f\right).
\end{align*}
Taking $\e>0$ small enough, we may assume that $S(\cdot,\f)$ has at most one critical point on $\supp \c_{\x,\f}$.

\noindent$(i)$
The proof is a small modification of that of \cite[Proposition 5.2]{TT}. 
We recall that the cut-off function $\chi$ is introduced in the begin of Subsection \ref{subsec:nonsta}.
By the support property of $\chi$,
\begin{align*}
|\pa_{\m}^{\a}\c_{x,\f}(\m)|\lesssim \m^{2d-\a}(1+x (\sin\f)\m)^{-d},\quad \supp \c_{x,\f}\subset (\m_0-2\e,\m_0+2\e)
\end{align*}
for $x\gg 1$, where we use $\n_m=\rho^{-1}m+O(1)$ as $m\to \infty$. Moreover,
\begin{align*}
\pa_{\m}^2S(\m,\f)=\s_1\frac{1}{\sqrt{1-\m^2}},\quad \pa_{\m}\pa_{\f}S(\m,\f)=\s_2\rho.
\end{align*}
The assumption $0\in\mathscr{C}_{\s,\rho}(\f_0,q_0)$ implies $(\pa_{\m}S)(0,\f_0)=0$. Therefore, the phase function $S(\m,\f)$ satisfies the assumption of Proposition \ref{Stphasemovecrit}\footnote{Strictly speaking, we should consider the variable $\tilde{\f}=\pi-\f$ to apply Proposition \ref{Stphasemovecrit} when $\f_0=\pi$.}. Applying Proposition \ref{Stphasemovecrit} with $\l=x$, we obtain
\begin{align*}
|I_{\s,q_0,\m_0}^{\e}(x,\f)|\leq x^{d+\frac{1}{2}}\left|\int_{\re}e^{ixS(\m,\f) }\c_{x,\f}(\m) d\m\right|\lesssim  x^{d+\frac{1}{2}}\cdot x^{-d-\frac{1}{2}}=1.
\end{align*}
This completes the proof.

\noindent$(ii)$ Since $\m_0\in (0,1)$, we can choose $\e>0$ sufficiently small such that $[\m_0-2\e,\m_0+2\e]\subset (0,1)$. By the support property of $\chi$, we have
\begin{align*}
|\pa_{\m}^{\a}\c_{x,\f_0}(\m)|\lesssim 1,\quad \supp \c_{x,\f_0}\subset [\m_0-2\e,\m_0+2\e],\quad \pa_{\m}^2S(\m_0,\f)=\frac{\s_1}{\sqrt{1-\m_0^2}}\neq 0
\end{align*}
and $\m_0$ is the only critical point of $S(\m,\f)$ in $\supp \c_{x,\f_0}$. Applying the stationary phase formula (of the form Lemma \cite[A.5]{TT}, see also \cite[p.334]{S}),
\begin{align*}
|I_{\s,q_0,0}^{\e}(x,\f)|\leq x^{d+\frac{1}{2}}\left|\int_{\re}e^{ixS(\m,\f) }\c_{x,\f}(\m) d\m\right|\lesssim  x^{d+\frac{1}{2}}\cdot x^{-\frac{1}{2}}=x^d,
\end{align*}
which proves \eqref{eq:baddecay}.

A more precise form of the stationary phase theorem \cite[Theorem 3.11]{Zwo} yields
\begin{align*}
I_{\s,q_0,0}^{\e}(x,\f_0)=&x^{d+\frac{1}{2}}\left(\frac{(2\pi)^{\frac{1}{2}}\c_{x,\f_0}(\m_0)}{x^{\frac{1}{2}}|(\pa_{\m}^2S)(\m_0,\f_0)|^{\frac{1}{2}} } e^{ixS(\m_0,\f_0)+\frac{\pi i}{4}(\pa_{\m}^2S)(\m_0,\f_0)}  +O(x^{-1})\right)\\
=&x^{d}\frac{(2\pi)^{\frac{1}{2}}\c_{x,\f_0}(\m_0)}{ |(\pa_{\m}^2S)(\m_0,\f_0)|^{\frac{1}{2}} } e^{ixS(\m_0,\f_0)+\frac{\pi i}{4}(\pa_{\m}^2S)(\m_0,\f_0)}  +O(x^{d-\frac{1}{2}})
\end{align*}
for $x\gg 1$. We have $\c_{x,\f_0}(\m_0)=\rho x^{-d+\frac{1}{2}}\z_{1,\s,q_0}\left(\rho x\m_0,x,\f_0\right)$ and $S(\m_0,\f_0)=\s_1\sqrt{1-\m_0^2}$ since $\m_0\in \mathscr{C}_{\s,\rho}(\f_0,q_0)$. By Lemma \ref{lem:amp} $(iii)$, we obtain \eqref{eq:I_1eachasym}.

\end{proof}

\begin{rem}
The reason why we obtain the improved decay $(i)$ when $0\in \mathscr{C}_{\s,\rho}(\f_0,q_0)$ holds is that the scaled amplitude function $\c_{x,\f}$ vanishes at $\m=0$ with order $2d$, which is the critical point of $\m\mapsto S(\m,\f_0)$.
\end{rem}

\begin{cor}\label{cor:I_1est}
Let $\s=(\s_1,\s_2)\in \{\pm\}\times \{\pm\}$. Then, the following statements hold true for sufficiently small $\e_0>0$:

\noindent$(i)$ Suppose that $\rho\geq 1$. Then, $|I_{1,\s}(x,\f)|\lesssim 1$ for $x\gg 1$ and $\f\in [0,\pi]$.

\noindent$(ii)$ $|I_{1,\s}(x,\f)|\lesssim 1$ for $x\gg 1$ and $\f\in [\e_0,\pi-\e_0]$.

\noindent$(iii)$ Suppose that $\rho> \frac{1}{2}$. Then, we have $|I_{1,\s}(x,\f)|\lesssim 1$ for $x\gg 1$ and $\f\in [0,\pi-\e_0]$.

\noindent$(iv)$ Suppose that $\rho^{-1}\notin 2\mathbb{N}$. Then $|I_{1,\s}(x,\f)|\lesssim x^d$ for $x\gg 1$ and $\f\in [0,\pi]$. Moreover, if $\f_0=0$ or $\f_0=\pi$, then
\begin{align*}
\sum_{\s_2\in \{\pm\}}I_{1,\s}(x,\f_0)=\sum_{\m\in \mathscr{D}_{\rho,\s_1}(\f_0)}  \frac{\rho^{2d+1}\m_0^{2d}}{d\Gamma(2d)}e^{i\s_1\sqrt{1-\m_0^2}x+iL_{\s_1,\rho,d,\m_0}}x^{d}+O(\max(x^{d-\frac{1}{2}},1))
\end{align*}
for $x\gg 1$, where $\mathscr{D}_{\rho,\s_1}(\f_0)$ is defined in \eqref{eq:phasecritrue} and $L_{\s_1,\rho,d,\m_0}\in \re$ is the same constant as in Lemma \ref{lem:amp} $(iii)$.

\end{cor}

\begin{proof}
The claims in $(i)-(iii)$ follow from Lemma \ref{lem:phasecript},  Propositions \ref{prop:Non-stationary} and \ref{prop:osciint}. The first inequality in $(iv)$ can be proved similarly if we take Lemma \ref{lem:phasecript} $(v)$ into account.
The point of the proof for $(i)$ is that the set of the critical points $\mathscr{C}_{\s,\rho}(0)\cup \mathscr{C}_{\s,\rho}(\pi)$ (for $\f=0,\pi$) consists of $\{0\}$ when $\rho\geq 1$, which lead to an improved decay of $I_{1,\s}$ by Proposition \ref{prop:osciint} $(i)$. Similarly, the part $(iii)$ follows from the fact that $\mathscr{C}_{\s,\rho}(0)$ (for $\f=0$) consists of $\{0\}$ when $\rho> \frac{1}{2}$.

It remains to prove the second result in $(iv)$. From Lemma \ref{lem:phasecript} $(v)$, Proposition \ref{prop:Non-stationary} and Proposition \ref{prop:osciint} $(i), (ii)$, we see
\begin{align*}
\sum_{\s_2\in \{\pm\}}I_{1,\s}(x,\f_0)=&\sum_{\s_2\in \{\pm\}}\sum_{q\in\mathbb{Z}}\sum_{\m_0\in \mathscr{C}_{\s,\rho}(\f_0,q)}I_{\s,q,\m_0}^{\e}(x,\f)+O(x^{-N})\\
=&\sum_{q\in\mathbb{Z}}\sum_{\m_0\in \mathscr{C}_{(\s_1,+),\rho}(\f_0,q)\cap (0,1)}(\rho^2\m_0)e^{i\s_1\sqrt{1-\m_0^2}x+iL_{\s_1,\rho,d,\m_0}}x^{-d+1}g_{d,\s_2}(\rho\m_0x,\f_0)\\
&+O(\max(x^{d-\frac{1}{2}},1))\\
=&\sum_{\m_0\in \mathscr{D}_{\rho,\s_1}(\f_0)}(\rho^2\m_0)e^{i\s_1\sqrt{1-\m_0^2}x+iL_{\s_1,\rho,d,\m_0}}x^{-d+1}(g_{d,+}+g_{d,-})(\rho\m_0x,\f_0)\\
&+O(\max(x^{d-\frac{1}{2}},1))
\end{align*}
where we use Lemma \ref{lem:phasecript} $(vi)$ in the last line. By Proposition \ref{Gegenasymp}, one has
\begin{align*}
(g_{d,+}+g_{d,-})(\rho\m_0x,\f_0)=\frac{(\rho\m_0)^{2d-1}}{d\Gamma(2d)}x^{2d-1}+O(x^{2d-2}).
\end{align*}
Combining these asymptotic expansions, we obtain the second result in $(iv)$.

\end{proof}

\subsection{Proof of the main results}

\begin{proof}[Proof of Theorems \ref{thm:dispersive} and \ref{thm:optimality}]
We recall $d=\frac{n-2}{2}$.
The results are direct consequences of the expression of the propagator \eqref{eq:propagatorsum} and \eqref{sIdecom}, and the estimates for $\sI_{\rho,d,c}, I_{1,\s}, I_{j,\s_2}, R$. More precisely, the estimate for $|\frac{r_1r_1}{2t}|\lesssim 1$ follows from Proposition \ref{prop:smallx} and that for $|\frac{r_1r_1}{2t}|\gg 1$ follows from Lemma \ref{lem:Rest}, Proposition \ref{prop:I_23}, and Corollary \ref{cor:I_1est}. The implications \eqref{eq:Dnonempty} directly follow from the definition of $\mathscr{D}_{\rho,\s_1}(\f_0)$.

\end{proof}

\appendix

\section{Miscellaneous}

\subsection{Asymptotic expansion of special functions}

We define
\begin{align*}
h_1(z)=\sqrt{1-z^2}-z\cos^{-1} z,
\end{align*}
where $\cos^{-1} z\in [0,\pi/2]$ for $0\leq z\leq 1$.

\begin{prop}\label{Besselasymp}

\noindent$(i)$ There are smooth functions $a_{\pm,x}:[0,x-\frac{1}{2}x^{\frac{1}{3}}]\to \mathbb{C}$ such that
\begin{align}\label{eq:Bessel1}
J_{\n}(x)=x^{-\frac{1}{4}}(x-\n)^{-\frac{1}{4}}(a_{+,x}(\n)e^{ixh_1(\frac{\n}{x})}+a_{-,x}(\n)e^{-ixh_1(\frac{\n}{x})} )
\end{align}
and for each $\a\in\mathbb{N}$, there exists $C_{\a}>0$ such that
\begin{align*}
|\pa_{\n}^{\a}a_{\pm,x}(\n)|\leq C_{\a}(x-\n)^{-\a},\quad \text{for}\quad  x\geq 8,\,\, \n\in [0,x-\frac{1}{2}x^{\frac{1}{3}}].
\end{align*}
Moreover, for fixed $\m_0\in (0,1)$,
\begin{align}\label{eq:Besselcoeff}
a_{\pm,x}(x\m_0)=\frac{e^{\mp \frac{\pi}{4}i}}{(2\pi)^{\frac{1}{2}}(1+\m_0)^{\frac{1}{4}}  }  +O(x^{-\frac{1}{2}})
\end{align}
for $x\gg 1$.

\noindent$(ii)$ Let $N>0$ and $\a\in\mathbb{N}$. Then there exist $C_{\a}>0$ and $C_{\a N}>0$ such that
\begin{align*}
|\pa_{\n}^{\a}J_{\n}(x)|\leq \begin{cases}C_{\a}x^{-\frac{1+\a}{3}} \quad \text{for}\quad x\geq 8,\,\, \n\in [x-2x^{\frac{1}{3}},x+2x^{\frac{1}{3}}],\\
C_{\a N}\n^{-\frac{1}{4}}(\n-x)^{-\frac{1}{4}-\a} (x^{-1}(\n-x)^3 )^{-N} \quad \text{for}\quad  x\geq 8,\,\, \n\in [x+\frac{1}{2}x^{\frac{1}{3}},\infty).
\end{cases}
\end{align*}
\end{prop}

\begin{proof}
The results except \eqref{eq:Besselcoeff} is proved in \cite[Proposition 3.1]{TT}.
The formula \eqref{eq:Besselcoeff} is perhaps more or less well-known. Here we give a self-contained proof.

By \cite[Lemma A.7]{TT}, there exists $\chi_1\in C_c^{\infty}((-\frac{3\pi}{4},\frac{3\pi}{4});[0,1])$ satisfying $\chi_1(w)=1$ for $|w|\leq \frac{2\pi}{3}$ and $\chi_1(w)=\chi_1(-w)$ such that
\begin{align*}
J_{\m_0x}(x)=\frac{1}{2\pi}\int_{\re}e^{ixS(w)}\chi_1(w)dw+O(x^{-\infty})
\end{align*}
where $S(w)=\sin w-\m_0w$. Since $\m_0\in (0,1)$, the critical points of $S$ in $\supp \chi_1$ is given by $w_{\pm}(\m_0):=\pm \cos^{-1}(\m_0)$. Moreover, one has
\begin{align*}
S(w_{\pm}(\m_0))=\pm h(\m_0),\quad  S''(w_{\pm}(\m_0))=\mp \sqrt{1-\m_0^2}.
\end{align*}
Applying the stationary phase formula of the form \cite[Theorem 3.11]{Zwo}, we obtain
\begin{align}\label{eq:Bessel2}
J_{\m_0x}(x)=\frac{1}{(2\pi)^{\frac{1}{2}}(1-\m_0^2)^{\frac{1}{4}} }(e^{ixh_1(\m_0)-\frac{\pi}{4}i}+e^{-ixh_1(\m_0)+\frac{\pi}{4}i})x^{-\frac{1}{2}}+O(x^{-1})
\end{align}
as $x\to \infty$. Then the formula \eqref{eq:Besselcoeff} follows from \eqref{eq:Bessel1}, \eqref{eq:Bessel2}, and the identity $x^{-\frac{1}{4}}(x-x\m_0)^{-\frac{1}{4}}=x^{-\frac{1}{2}}(1-\m_0)^{-\frac{1}{4}}$.

\end{proof}

\begin{prop}\label{Gegenasymp}
Let $d> 0$.
There are functions $g_{d,+}(m,\f), g_{d,-}(m,\f)$ which are smooth respect to $m\in [1,\infty)$ and a function $r(m,\f)$ defined for $m\in\mathbb{N}^*$ such that
\begin{align}\label{eq:Gegen1}
d^{-1}C_m^{d}(\cos\f)=\sum_{\pm}g_{d,\pm}(m,\f)e^{\pm im\f}+r(m,\f)
\end{align}
for $m\in \mathbb{N}$ and 
\begin{align*}
&|\pa_m^{\a}g_{d,\pm}(m,\f)|\leq 
C_{\a,d}m^{2d-1-\a}(1+m\sin\f)^{-d}\quad  m\geq 1\\
&|r(m,\f)|\leq C_{N,d}m^{-N}\quad  m\in\mathbb{N}^*
\end{align*} 
for each $N>0$, $\a\in\mathbb{N}$, and $\f\in [0,\pi]$ with constants $C_{\a,d}, C_{N,d}>0$. Moreover, for $\f_0=0$ or $\f_0=\pi$,
\begin{align}
g_{d,+}(m,\f_0)+g_{d,-}(m,\f_0)=\frac{1}{d\Gamma(2d)}m^{2d-1}+O(m^{2d-2})\label{eq:Gegen2}
\end{align}
as $m\to \infty$.

\end{prop}

\begin{proof}
The results except \eqref{eq:Gegen2} is proved in \cite[Proposition 3.3]{TT}.

We note
$C_m^{d}(1)=\frac{\Gamma(m+2d)}{\Gamma(m+1)\Gamma(2d)}=\frac{\Gamma(m+2d)}{m\Gamma(m)\Gamma(2d)}=\frac{1}{mB(m,2d)}$, where $\Gamma$ and $B$ denote the gamma function and the beta function respectively. Moreover, as is well-known (as a consequence of the Stirling formula), we have $B(m,2d)=\Gamma(2d)m^{-2d}+O(m^{-2d-1})$ as $m\to \infty$. Therefore,
\begin{align}\label{eq:Gegen4}
C_m^{d}(1)=\frac{1}{\Gamma(2d)}m^{2d-1}+O(m^{2d-2})\quad (\text{as}\quad m\to \infty).
\end{align}
Then, \eqref{eq:Gegen2} follows from  \eqref{eq:Gegen1}, \eqref{eq:Gegen4}, and the identity $C_m^{d}(-1)=(-1)^mC_m^{d}(1)$.

\end{proof}

In particular, we have a uniform bound:
\begin{align}
|J_{\n}(x)|\leq& Cx^{-\frac{1}{3}}\quad x\geq 8,\,\, \m\geq 0 \label{eq:Besselunif}\\
|d^{-1}C_m^{d}(\cos \f)|\leq& Cm^{2\n-1}\quad m\in \mathbb{N}^*,\,\, \f\in[0,\pi]. \label{eq:Gegenunif}
\end{align}

\subsection{Stationary phase formula}
We consider the following integral with parameters $\l,\f$:
\begin{align*}
I(\l,\f):=\int_{\re}e^{i\l S(\m,\f)}\c(\m,\l,\f)d\m
\end{align*}
and its decay rate with respect to $\l\gg 1$. The next proposition is a generalization of \cite[Proposition 2.4]{TT}.

\begin{prop}\label{Stphasemovecrit}
Let $d\geq 0$ and $c, C>0$. Suppose that $S\in C^{\infty}([0,1]^2;\re)$ and $\c(\cdot,\l,\cdot)\in C^{\infty}(\re\times [0,1])$ satisfy 
\begin{align*}
&\supp \c(\cdot ,\l,\f)\subset [0,1],\quad |\pa_{\m}^{\a}\c(\m,\l,\f)|\leq C_{\a}\m^{2d-\a}(1+\l \f\m)^{-d},\\ 
&
|\pa_{\m}^2S(\m,\f)|\geq c,\quad |\pa_{\m}\pa_{\f}S(\m,\f)|\geq C,\quad (\pa_{\m}S)(0,0)=0
\end{align*}
uniformly in $\m\in [0,1]$, $\f\in [0,1]$ and $\l\gg 1$. 
Then
 $|I(\l,\f)|\lesssim \l^{-d-\frac{1}{2}}$ uniformly in $\l\gg 1$ and $\f\in [0,1]$.
\end{prop}

\begin{rem}
A typical example of phase functions satisfying the assumption of \cite[Proposition 2.4]{TT} is $S(\m,\f)=(\m-\f)^2$. The assumption here includes examples such as $S(\m,\f)=(\m+\f)^2$ although the proof is much easier than the former case.

\end{rem}

\begin{proof}
We may assume $S(0,0)=0$.
We extend $S$ to a smooth function defined near $(0,0)$. By the implicit function theorem, we find a unique smooth function $\m(\f)$ such that
\begin{align*}
(\pa_{\m}S)(\m,\f)=0 \Leftrightarrow \m=\m(\f),\quad \m(0)=0.
\end{align*}
Differentiating the equation $(\pa_{\m}S)(\m(\f),\f)=0$, we have $\m'(0)=-(\pa_{\m}\pa_{\f}S)(0,0)/(\pa_{\m}^2S)(0,0)$. Then, we find that 
\begin{align*}
(\pa_{\m}\pa_{\f}S)(0,0)/(\pa_{\m}^2S)(0,0)>0\Rightarrow \m(\f)\geq 0\quad \text{for}\quad \f\in [0,\pi]\\
(\pa_{\m}\pa_{\f}S)(0,0)/(\pa_{\m}^2S)(0,0)<0\Rightarrow \m(\f)\leq 0\quad \text{for}\quad \f\in [0,\pi].
\end{align*}
In the former case, the result was proved in \cite[Proposition 2.4]{TT}. Therefore, we consider the latter case only. This implies that for $(\m,\f)\in [0,1]^2$, $(\pa_{\m}S)(\m,\f)=0$ if and only if $\m=\f=0$ and that
\begin{align*}
|\pa_{\m}S(\m,\f)|\gtrsim |\m+\f|\quad \text{for} \quad (\m,\f)\in [0,1]^2.
\end{align*}

By scaling, we have
\begin{align*}
I(\l,\f)=\f\l^{-d}\int_{\re}e^{i\l \f^2S_{\f}(\m)}\c_{\f}(\m,\l)d\m,
\end{align*}
where we set
\begin{align*}
S_{\f}(\m)=\f^{-2}S(\f\m,\f),\quad \c_{\f}(\m,\l)=\l^{d}\c(\f\m,\l,\f).
\end{align*}
Then
\begin{align*}
|\pa_{\m}S_{\f}(\m)|=\f^{-1}\cdot |(\pa_{\m}S)(\f\m,\f)|\gtrsim |\m+1|
\end{align*}
for $\m\in \supp \c_{\f}(\cdot ,\l)$ and $\f\in (0,1]$. 
Moreover, for $\a\in\mathbb{N}$ and $\a'\in \mathbb{N}\setminus\{0\}$, we have
\begin{align*}
|\pa_{\m}^{\a}\c_{\f}(\m,\l)|=&\l^{d} \f^{\a}|(\pa_{\m}^{\a}\c)(\f\m,\l,\f)|\\
\lesssim& \l^{d}\f^{\a} (\f \m)^{2d-\a}(1+\l \f^2 \m)^{-d}\lesssim \m^{d-\a},\\
\supp (\c_{\f}(\cdot,\l))\subset& \{\f^{-1}\l^{-\frac{1}{2}}\leq \m\leq \f^{-1}\},\quad |\pa_{\m}^{1+\a'}S_{\f}(\m)|\lesssim |\f|^{\a'-1}
\end{align*}
By integrating by parts $N(>d+1)$ times (for the integrability of the above integral), we have
\begin{align*}
|I(\l,\f)|\lesssim \f\l^{-d}\cdot (1+\l \f^2)^{-N}\lesssim \f\l^{-d}\cdot (1+\l \f^2)^{-\frac{1}{2}}\lesssim \l^{-\frac{1}{2}-d}.
\end{align*}
This completes the proof.

\end{proof}

\subsection{Estimates for $I_{2,\s_2}$ and $I_{3,\s_2}$}\label{subsec:I_23}

We recall $S_{\s_2}(m,\f)=\s_2\f m-\frac{\pi}{2\rho}m$.

\begin{lem}

Let $\s=(\s_1,\s_2)\in \{\pm\}\times\{\pm\}$, $N>0$, $0<\e<\pi/2$ and $\a\in\mathbb{N}$. 

\noindent$(i)$ For $\a\geq 2$, we have $\pa_m^{\a}S_{\s_2}(m,\f)=0$.

\noindent$(ii)$ $\supp \z_{1,\s}(\cdot,x,\f)\subset \Omega_1$ and $\supp \z_{j,\s_2}(\cdot,x,\f)\subset \Omega_j$ for $j=2,3$.

\noindent$(iii)$ We have 
\begin{align*}
|\z_{2,\s_2}(m,x,\f)|\lesssim \begin{cases}x^{-\frac{1}{3}}\quad \text{for}\quad \f\in [\e,\pi-\e]\\
x^{d-\frac{1}{3}}\quad \text{for}\quad \f\in [0,\pi]
\end{cases}
\end{align*}
and for $m\geq 1$, $x\gg 1$. Moreover, 
\begin{align*}
|\pa_{m}^{\a}\z_{2,\s_2}(m,x,\f)|\lesssim x^{d-\frac{1+\a}{3}} 
\end{align*}
for $m\geq 1$, $x\gg 1$ and $\f\in [0,\pi]$.

\noindent$(iv)$
We have
\begin{align*}
|\z_{3,\s_2}(m,x,\f)|\lesssim \begin{cases}x^{d+\frac{1}{12}}(\n_m-x)^{-\frac{5}{4}}\quad \text{for}\quad x+\frac{1}{2}x^{\frac{1}{3}}\leq \n_m\leq 4x,\,\, \f\in [0,\pi]\\
x^{\frac{1}{12}}(\n_m-x)^{-\frac{5}{4}}\quad \text{for}\quad x+\frac{1}{2}x^{\frac{1}{3}}\leq \n_m\leq 4x,\,\, \f\in [\e,\pi-\e]\\
(1+m)^{-N}\quad \text{for}\quad  \n_m\geq 2x,\,\, \f\in [0,\pi]
\end{cases},
\end{align*}
and for $m\geq 1$, $x\gg 1$. Moreover, 
\begin{align*}
|\pa_{m}^{\a}\z_{3,\s_2}(m,x,\f)|\lesssim x^{d-\frac{1+\a}{3}}
\end{align*}
for $m\geq 1$ $x\gg 1$ and $\f\in [0,\pi]$.
\end{lem}

This lemma directly follows from the definition of $S_{\s}$ and $\z_{j,\s_2}$ ($j=2,3$) given in Subsection \ref{subsec:decom}, and Propositions \ref{Besselasymp}, \ref{Gegenasymp}.

\begin{proof}[Proof of Proposition $\ref{prop:I_23}$]

The proof is almost identical to that of \cite[Propositions 4.1, 4.2]{TT}.
We recall $\n_m=\sqrt{\rho^{-2}m(m+2d)+d^2+c}$ and $\n_m\sim \rho^{-1}m$ for $m\gg 1$.
First, we deal with the case $j=2$.

$(i)$
By Lemma \ref{lem:amp} $(i)$ and $(iii)$, 
\begin{align}\label{eq:I_2pf1}
|I_{2,\s_2}(y,\f)|\lesssim \left\{\begin{aligned}
&x^{-\frac{1}{3}}\sum_{m\in \Omega_2\cap\mathbb{N}}1\lesssim 1\quad &&(\f\in [\e,\pi-\e])\\
&x^{d-\frac{1}{3}}\sum_{m\in \Omega_2\cap\mathbb{N}}1\lesssim x^d\quad &&(\f\in [0,\pi])
\end{aligned}\right.
\end{align}
where we recall $\Omega_2=\{m\in\re\mid  x-2x^{\frac{1}{3}}\leq \n_m\leq x+2x^{\frac{1}{3}}\}$ and use the number of element of $\Omega_2\cap \mathbb{N}$ is bounded by $x^{\frac{1}{3}}$ times a constant. The first estimate of \eqref{eq:I_2pf1} gives $(i)$ for $j=2$.

\noindent$(ii)$ Suppose that $\rho^{-1}\notin2\pi\mathbb{Z}$. Then, the assumption of \cite[Proposition 2.2]{TT} is satisfied by the last lemma with $k=d-\frac{1}{3}$, $M \sim x$ and $\rho=\frac{1}{3}$. Thus, $(ii)$ follows for $j=2$.

\noindent$(iii)$ This part for $j=2$ follows from \eqref{eq:I_2pf1} and $(ii)$.

Next, we deal with the case $j=3$.

\noindent$(i)$
 Taking $\chi\in C^{\infty}(\re;[0,1])$ such that $\chi(\n)=1$ for $\n\leq 2$ and $\chi(\n)=0$ for $\n\geq 4$ and setting $\overline{\chi}=1-\chi$, we write
\begin{align*}
I_{3,\s_2}(x,\f)=&\left(\sum_{m=1}^{\infty}(\chi(\n_m/x)+\overline{\chi}(\n_m/x)) \z_{3,\s_2}(m,x,\f)e^{iS_{\s}(m,\f)} \right)\\
=:&I_{3,1,\s_2}(x,\f)+I_{3,2,\s_2}(x,\f).
\end{align*}
Since $\overline{\chi}(\n_m/x)\z_{3,\s_2}(m,x,\f)$ is rapidly decreasing with respect to $m$ by Lemma \ref{lem:amp} $(iv)$, we have
\begin{align}\label{pf:I_3est1}
|I_{3,2,\s_2}(x,\f)|\lesssim \sum_{\n_m\geq 2x,m\geq 1}^{\infty}(1+m)^{-N-1}\lesssim x^{-N}.
\end{align}
Lemma \ref{lem:amp} $(iv)$ implies
\begin{align}\label{pf:I_3est2}
|I_{3,1,\s_2}(x,\f)|\lesssim& \left\{\begin{aligned}
&x^{d}\sum_{\n_m\in [x+\frac{1}{2}x^{\frac{1}{3}},4x]}(\n_m-x)^{-\frac{5}{4}}\lesssim 1\quad &&(\f\in [\e,\pi-\e])\\
&x^{d+\frac{1}{12}}\sum_{\n_m\in [x+\frac{1}{2}x^{\frac{1}{3}},4x]}(\n_m-x)^{-\frac{5}{4}}\lesssim x^d\quad &&(\f\in [0,\pi])
\end{aligned}\right..
\end{align}
The inequalities \eqref{pf:I_3est1} and \eqref{pf:I_3est2} prove $(i)$ for $j=3$.

\noindent$(ii)$ Suppose that $\rho^{-1}\notin2\pi\mathbb{Z}$. Then, the assumption of \cite[Proposition 2.2]{TT} is satisfied by the last lemma with $k=d-\frac{1}{3}$, $M \sim x$ and $\rho=\frac{1}{3}$. Thus, we obtain $(ii)$ for $j=3$. The part $(iii)$ for $j=3$ also follows from \eqref{pf:I_3est1}, \eqref{pf:I_3est2}, and the part $(ii)$. This completes the proof.

\end{proof}

\subsection{Estimates of $I_{1,\s}$ away from the conjugate pairs}\label{subsec:I_1awayconj}

\begin{proof}[Proof of Proposition \ref{prop:awayfromconj}]
The proof is identical to that of \cite[Proposition 5.1]{TT}. We recall
\begin{align*}
I_{\s,q_0,\m_0}^{\e}(x,\f):=\int_{\re}\chi_{\e,\m_0}(m) \z_{1,\s,q_0}(m,x,\f)e^{i(S_{1,\s}(m,x,\f)+2\pi q_0m)}dm.
\end{align*}
By the change of variable $m=\rho x\m$,  
\begin{align}\label{I_qscale}
I_{\s,q_0,\m_0}^{\e}(x,\f)=&x^{\frac{1}{2}}\int_{\re}e^{ixS(\m,\f) }\c_{y,\f}(\m) d\m,
\end{align}
where we set
\begin{align*}
S(\m,\f)=\s_1h_1(\m) +\left(\s_2\rho\f-\frac{\pi}{2}+2\pi\rho q_0\right)\m ,\quad \c_{x,\f}(\m)=x^{\frac{1}{2}}\chi_{\e,\m_0}(\rho x\m) \z_{1,\s,q_0}(\rho x\m,x,\f).
\end{align*}
Thus it remains to show 
\begin{align}\label{I_1stphase}
\left|\int_{\re}e^{ixS(\m,\f) }\c_{x,\f}(\m) d\m\right|\lesssim x^{-\frac{1}{2}}\quad \text{for}\quad x\gg 1,\,\, \f\in [0,\pi].
\end{align}
By Lemma \ref{lem:amp} $(i)$ and $(ii)$, we have
\begin{align}\label{I_qgamest}
|\pa_\m^{\a}\c_{x,\f}(\m)|\lesssim \m^{d-\a}(1-\m)^{-\frac{1}{4}-\a},\quad \supp \c_{x,\f}\subset \{\m\in\re\mid 0\leq \m\leq 1-\frac{1}{3}x^{-\frac{2}{3}} \}
\end{align}
for $x\gg 1$ and $\f\in [0,\pi]$, where we use $\n_m=\rho^{-1}m+O(1)$ as $m\to \infty$.

We write
\begin{align*}
\int_{\re}e^{ixS(\m,\f) }\c_{x,\f}(\m)d\m=\int_{\re}e^{ixS(\m,\f) }\c_{x,\f}(\m)\g_1(\m)d\m+\int_{\re}e^{ixS(\m,\f) }\c_{y,\f}(\m)\g_2(\m)d\m=I_1+I_2,
\end{align*}
where $\g_1,\g_2\in C_c^{\infty}(\re;[0,1])$ satisfy $\g_1(\m)+\g_2(\m)=1$ for $\m\in [0,1]$, $\g_1(\m)=1$ for $0\leq \m\leq 1-2\d$ and $\g_2(\m)=1$ for $1-\d\leq \m\leq 1$ where $\d>0$ is determined later. 
Since
\begin{align*}
|\pa_{\m}^2S(\m,\f)|=|\s_1\pa_{\m}^2h_1(\m)|=|(1-\m^2)^{-\frac{1}{2}}|\ge 1\quad \text{for}\quad \m\in \supp \g_1\cap[0,1],
\end{align*}
the stationary phase theorem (\cite[p.334]{S} or \cite[Lemma A.5]{TT}) implies $|I_1|\lesssim x^{-\frac{1}{2}}$.

On the other hand, using the change of variable $\m'=\sqrt{1-\m}$ (with $\m=1-\m'^2$), we have $d\m=-2\m'd\m'$ and
\begin{align*}
I_2=\int_{\re}e^{ix\tilde{S}(\m',\f) }\c_2(\m')d\m'
\end{align*}
where we set $\tilde{S}(\m',\f)=S(1-\m'^2,\f)$ and $\c_2(\m')=2\m'\c_{x,\f}(1-\m'^2)\g_2(1-\m'^2)$. We note that 
\begin{align*}
|\pa_{\m'}^{\a}\c_2(\m')|\lesssim\m'^{\frac{1}{2}-\a},\quad\supp\c_2\subset \{  \frac{1}{\sqrt{2}}x^{-\frac{1}{3}}\leq \m'\leq \sqrt{2\d}\}
\end{align*}
for $x\gg 1$ and $\f\in [0,\pi]$ and that $\tilde{S}_q$ is smooth with respect to $\m'$ close to $0$ and $\f\in[0,\pi]$.
It follows from the identity  
\begin{align*}
\pa_{\m'}^3\tilde{S}(\m',\f)=&\s_1\pa_{\m'}^3(h_1(1-\m'^2))=12\s_1\m' h_1''(1-\m'^2)-8\s_1\m'^3h_1^{(3)}(1-\m'^2)\\ 
=&4\sqrt{2}\s_1+O(\m')\quad \text{as}\quad \m'\to 0
\end{align*}
that $|\pa_{\m'}^3\tilde{S}(\m',\f)|\gtrsim 1$ for $\m'\in \supp \c_2$ and $\f\in [0,\pi]$ if $\d>0$ is small enough\footnote{We note that $h(1-\m'^2)$ is smooth at $\m'=0$ although $h(\m)$ is not smooth at $\m=1$}. Thus, the (degenerate) stationary phase theorem (\cite[p.334]{S} or \cite[Lemma A.5]{TT}) implies\footnote{Here we use $\c_2(\m')=O((\m')^{\frac{1}{2}})$ as $\m'\to 0$.}
\begin{align*} 
|I_2|\lesssim (x^{-\frac{1}{3}})^{1+\frac{1}{2}} =x^{-\frac{1}{2}}.
\end{align*}
This proves \eqref{I_1stphase} and completes the proof of Proposition \ref{prop:awayfromconj}.

\end{proof}


\begin{thebibliography}{99}

\bibitem{C} J. Cheeger. On the spectral geometry of spaces with cone-like singularities. Proceedings of the National Academy of Sciences, \textbf{76}(5), (1979), 2103--2106.

\bibitem{FFFP1} L. Fanelli, V. Felli, M.A. Fontelos, A. Primo, Time decay of scaling critical electromagnetic Schr\"odinger flows. Comm. Math. Phys. \textbf{324} (2013), no.3, 1033--1067.

\bibitem{FFFP2} L. Fanelli, V. Felli, M.A. Fontelos, A. Primo, Time decay of scaling invariant electromagnetic Schr\"odinger equations on the plane. Comm. Math. Phys. \textbf{337} (2015), no.3, 1515--1533.



\bibitem{F} G. Ford, The Fundamental Solution and Strichartz Estimates for the Schr\"odinger Equation on Flat Euclidean Cones. Commun. Math. Phys. \textbf{299} (2010), 447--467. 



\bibitem{HZ} A. Hassell, J. Zhang, Global-in-time Strichartz estimates on nontrapping asymptotically conic manifolds, Anal.PDE, \textbf{9} (2016), 151--192.


\bibitem{JK} A. Jensen, T. Kato, Spectral properties of Schr\"odinger operators and time-decay of wave functions. Duke Math J \textbf{46}, 583--611 (1979).

\bibitem{JZ} Q. Jia, J. Zhang, Pointwise dispersive estimates for Schr\"odinger and wave equations in a conical singular space, preprint, https://arxiv.org/abs/2411.16029.


\bibitem{KT} M. Keel, T. Tao, Endpoint Strichartz estimates. Amer. J. Math. \textbf{120} (1998), no. 5, 955--980.

\bibitem{KM} B. Keeler, J. L. Marzuola, Pointwise dispersive estimates for Schr\"odinger operators
on product cones, J. Diff. Equ., \textbf{320} (2022), 419--468.








\bibitem{S} E. M. Stein, Harmonic analysis: real-variable methods, orthogonality, and oscillatory integrals. With the assistance of T. S. Murphy. Princeton Mathematical Series, 43. Monographs in Harmonic Analysis, III. Princeton University Press, Princeton, N.J., 1993.





\bibitem{TT} K. Taira, H. Tamori, Strichartz estimates for the (k,a)-generalized Laguerre operators, SIGMA \textbf{21} (2025), 014, 37 pages.

\bibitem{W} X. Wang, Asymptotic expansion in time of the Schr\"odinger group on conical manifolds, Ann. Inst.
Fourier, \textbf{56} (2006), 1903--1945.


\bibitem{Z} J. Zhang, Resolvent and spectral measure for Schr\"odinger operators on flat Euclidean cones,
J. Funct. Anal. \textbf{282} (2022), 109311.


\bibitem{ZZ0} J. Zhang and J. Zheng, Global-in-time Strichartz estimates and cubic Schr\"odinger equation
in a conical singular space, arXiv:1702.05813.




\bibitem{Zwo} M. Zworski, Semiclassical analysis. Graduate Studies in Mathematics, 138. American Mathematical Society, Providence, RI, (2012).

\end{thebibliography}
\end{document}